\documentclass[12pt]{amsart}

\usepackage{color}

\usepackage{graphicx}
\usepackage{amsmath,amsthm}
\usepackage{amsfonts}
\usepackage{amssymb}
\usepackage{a4wide}
\title{Enumeration of $\C{H}$-strata in quantum matrices with respect to dimension}
\author{J.~Bell, K.~Casteels and S.~Launois}
\thanks{The first and second named authors thank NSERC for its generous support.  The third named author's research was supported by a Marie Curie European Reintegration Grant within the $7^{\mbox{th}}$ European Community Framework Programme}
\keywords{}
\subjclass[2000]{16W35; 20G42}

\address{Jason Bell\\
Department of Mathematics\\
Simon Fraser University\\
Burnaby, BC V5A 1S6, Canada
}

\email{jpb@math.sfu.ca}

\address{Karel Casteels\\
Department of Mathematics\\
University of California\\
Santa Barbara, CA 93106
}

\email{casteels@math.ucsb.edu}

\address{St\'ephane Launois \\
School of Mathematics, Statistics \& Actuarial science \\
University of Kent\\
Canterbury, Kent CT2 7NF, United Kingdom}

\email{S.Launois@kent.ac.uk}

\newtheorem{thm}{Theorem}[section]
\newtheorem{lem}[thm]{Lemma}

\newtheorem{prop}[thm]{Proposition}
\newtheorem{cor}[thm]{Corollary}

\theoremstyle{definition}
\newtheorem{defn}[thm]{Definition}

\newtheorem{rem}{Remark}

\newtheorem{ex}[thm]{Example}

\newtheorem{notn}[thm]{Notation}

\newcommand{\C}{\mathcal}
\newcommand{\F}{\mathfrak}
\newcommand{\Oq}{
\mathcal{O}_q(M_{m,n}(\mathbb{K}))
}

\newcommand{\lef}[1]{
\textnormal{left}(#1)
}

\newcommand{\up}[1]{
\textnormal{up}(#1)
}

\newcommand{\hspec}{
\mathcal{H}\textnormal{-spec}(\Oq)
}
\newcommand{\vb}[1]{
{\boldsymbol #1}
}

\begin{document}
\maketitle
\date
\abstract{We present a combinatorial method to determine the dimension of $\C{H}$-strata in the algebra of $m\times n$ quantum matrices $\Oq$ as follows. To a given $\C{H}$-stratum we associate a certain permutation via the notion of pipe-dreams. We show that the dimension of the $\C{H}$-stratum is precisely the number of odd cycles in this permutation. Using this result, we are able to give closed formulas for the trivariate generating function that counts the $d$-dimensional $\C{H}$-strata in $\Oq$. Finally, we extract the coefficients of this generating function in order to settle conjectures proposed by the first and third named authors~\cite{bldim,bll} regarding the asymptotic proportion of $d$-dimensional $\C{H}$-strata in $\Oq$. }
\section{Introduction}

For positive integers $m$ and $n$ and for $q$ a nonzero element of a field $\mathbb{K}$ that is not a root of unity, let us denote by $\C{A}=\Oq$ the algebra of $m\times n$ quantum matrices. There is a natural action of the algebraic torus $\C{H}=(\mathbb{K}^*)^{m+n}$ on $\C{A}$ which, by work of Goodearl and Letzter \cite{bg} allows the prime spectrum of $\C{A}$ to be partitioned into a finite number of disjoint \emph{$\C{H}$-strata}. Moreover, each $\C{H}$-stratum is homeomorphic (with respect to the Zariski topology) to the prime spectrum of a commutative Laurent polynomial ring over $\mathbb{K}$. 

In this work, we complete the project started in~\cite{bln} and continued in~\cite{bldim,bll}; namely, that of determining a useful condition to determine the dimension of a given $\C{H}$-stratum. Furthermore, this condition enables one to easily enumerate the $\C{H}$-strata in $\Oq$ with respect to dimension. The principal motivation for this originates in Dixmier's idea that for an infinite-dimensional algebra, identifying the primitive ideals forms an important first step towards understanding the representation theory of the algebra. On the other hand, as a consequence of the $\C{H}$-stratification theory, the primitive ideals are those prime ideals that are maximal within their $\C{H}$-stratum. In particular, primitive $\C{H}$-primes correspond to zero-dimensional $\C{H}$-strata.

Our condition is roughly described as follows. Within each $\C{H}$-stratum there is a unique prime ideal that is invariant under the action of $\C{H}$, a so-called \emph{$\C{H}$-prime}. Next, to any given $\C{H}$-prime, we may associate a certain permutation $\tau$, which for reasons that will become clear, we call a \emph{toric permutation}. Our first main result is the following theorem.
\begin{thm} \label{maintheorem}
Let $J$ be an $\C{H}$-prime in $\Oq$ and let $\tau$ be the associated toric permutation. The dimension of the $\C{H}$-stratum containing $J$ is precisely the number of odd cycles in the disjoint cycle decomposition of $\tau$.
\end{thm}
To clarify, the parity of a cycle is defined to be the parity of the number of inversions. Thus a cycle is odd if and only if it has even length. 

The crucial point in the proof of Theorem~\ref{maintheorem} is an isomorphism given in Theorem~\ref{main2} between the kernels of certain linear maps. Before describing this, we note that the authors~\cite{bcl}, and independently Yakimov \cite{yakimov2} with stronger hypotheses on $\mathbb{K}$ and $q$, have generalized the isomorphism to obtain a formula for calculating the dimension of torus-invariant strata in certain important subalgebras of the quantized enveloping algebra $U_q(\F{g})$ of a simple complex Lie algebra $\F{g}$. In this work we give a proof in the context of quantum matrices both for completeness and for its interesting combinatorial nature.

The proof of Theorem~\ref{maintheorem} depends on two parametrizations of the set of $\C{H}$-primes of $\Oq$. The first, due to Cauchon~\cite{cauchon1}, assigns to every $\C{H}$-prime a combinatorial object called a \emph{Cauchon diagram}. This is simply an $m\times n$ grid of squares coloured black or white according to the rule that if a square is black, then either all squares strictly above or all squares strictly to the left are also black. See Figure~\ref{cauchonexamples} for some examples. Interestingly, Cauchon diagrams appear independently in the work of Postnikov~\cite{postnikov} under the name \reflectbox{L}-diagrams (or ``le''-diagrams), as a parametrization of the totally nonnegative Grassmann cells of the totally nonnegative Grassmannian.  The connection between quantum matrices and total nonnegativity is detailed in papers of Goodearl, Launois and Lenagan~\cite{gll,gll2,gll3}.

The second parametrization of $\C{H}$-primes consist of the set of \emph{restricted ($m+n$)-permutations}, that is, permutations $\sigma\in S_{m+n}$ that satisfy $$-n\leq \sigma(i)-i\leq m$$ for all $i\in\{1,\ldots,m+n\}$. In~\cite{launois1} it has been shown that the set of restricted ($m+n$)-permutations ordered under the reverse Bruhat order is order-isomorphic to the poset of $\C{H}$-primes ordered by inclusion. This result was recently generalized by Yakimov \cite{yakimov}. A bijection between the set of Cauchon diagrams and the set of restricted permutations can be made using \emph{pipe dreams} (see Section~\ref{pipedreams}).

Bell and Launois~\cite{bldim} have shown that the dimension of a given $\C{H}$-stratum may be computed using the associated Cauchon diagram $D$. In particular, they show that one may construct a certain skew-symmetric matrix $M(D)$ from $D$ such that the dimension of the $\C{H}$-stratum is exactly $\dim(\ker(M(D))$. 

We let $\omega$ denote the maximum element in the poset of restricted $(m+n)$-permutations, and suppose that we have a Cauchon diagram $D$ whose corresponding restricted permutation is $\sigma$. The \emph{toric permutation} corresponding to $D$ is defined to be the permutation $\tau=\sigma\omega^{-1}$. In Theorem~\ref{main2} we construct an isomorphism between $\ker(M(D))$ and $\ker(P_\sigma+P_\omega)$, where $P_{\mu}$ is the matrix representation of a permutation $\mu$. As this latter space has dimension equal to the number of odd cycles of $\tau$, we obtain Theorem~\ref{maintheorem}.

As an application we are able to show in Corollary~\ref{enumerationcor} that the number of $d$-dimensional $\C{H}$-strata in $\Oq$ is the coefficient of $\frac{x^m}{m!}\frac{y^n}{n!}{t^d}$ in the power series expansion of $$(e^{-y}+e^{-x}-1)^\frac{-1-t}{2}(e^x+e^y-1)^\frac{1-t}{2}.$$ 

By determining the coefficients of this power series, we are able to settle several conjectures from~\cite{bll,bldim} concerning the asymptotic proportion of $d$-dimensional $\C{H}$-strata in $\Oq$. Namely, we prove in Theorem~\ref{egf2cor} that for fixed $m$ and $d$, the proportion of $d$-dimensional $\C{H}$-strata in $\Oq$ tends to $\frac{a(d)}{m!2^m}$ as $n\rightarrow\infty$, where $a(d)$ is the coefficient of $t^d$ in the polynomial $(t+1)(t+3)\cdots (t+2m-1)$.

\section{Preliminaries}
In this section, we give some basic background on quantum matrices, the Goodearl-Letzter stratification theory, and Cauchon diagrams.
\subsection{Quantum Matrices}
Throughout this paper, we set $\mathbb{K}$ to be a field, $q$ is a nonzero element of $\mathbb{K}$ that is not a root of unity, and we fix two positive integers $m$ and $n$. For a positive integer $\ell$, let $[\ell]:=\{1,2,\ldots, \ell\}$. 

\begin{defn} We let $\Oq$ denote the \emph{quantized coordinate ring of $m\times n$ matrices}. This is the algebra with generators $x_{i,j}$ for all $(i,j)\in [m]\times [n]$, subject to  the following relations:
\begin{enumerate}
\item For all $i\in[m]$ and $j,k\in[n]$ with $j<k$, $$x_{i,j}x_{i,k}=qx_{i,k}x_{i,j};$$
\item For all $j\in[n]$ and $i,\ell \in[m]$ with $i<\ell$, $$x_{i,j}x_{\ell,j}=qx_{\ell,j}x_{i,j};$$
\item For all $i,\ell\in [m]$ with $i<\ell$ and distinct $j,k\in [n]$:
\begin{displaymath}
 x_{i,j}x_{\ell,k} = \left\{\begin{array}{ll}
x_{\ell,k}x_{i,j}, & \textnormal{ if $j>k$;} \\
x_{\ell,k}x_{i,j} +(q-q^{-1})x_{i,k}x_{\ell,j}, & \textnormal{ if $j<k$.}
\end{array} \right.
\end{displaymath}
 \end{enumerate}
 \end{defn}
 
 The algebra $\Oq$ is colloquially known as \emph{the algebra of $m\times n$ quantum matrices}, or just \emph{quantum matrices}. It is well-known that $\Oq$ may be presented as an iterated Ore extension of the base field. We set $X$ to be the matrix of generators obtained by setting $X[i,j]:=x_{i,j}$. The collection of prime ideals of $\Oq$ is the \emph{prime spectrum} and is denoted by ${\rm spec}(\Oq)$. We endow the prime spectrum with the Zariski topology. As we assume that $q$ is not a root of unity, every prime ideal of $\Oq$ is completely prime \cite{GL2}.
 
 \subsection{$\C{H}$-Stratification Theory}
We wish to understand the structure of ${\rm spec}(\Oq)$. To this end, a useful tool is the $\C{H}$-stratification of Goodearl and Letzter. First notice that the algebraic torus $\C{H}=(\mathbb{K}^*)^{m+n}$ acts rationally by automorphisms on $\Oq$ as follows. If $$h=(\rho_1,\ldots,\rho_m,\gamma_1,\ldots,\gamma_n)\in\C{H},$$ then for all $(i,j)\in [m]\times [n]$ we set $$h\cdot x_{i,j} := \rho_i\gamma_jx_{i,j}.$$ 

 An ideal $J$ is an \emph{$\C{H}$-ideal} if $h\cdot J = J$ for all $h\in\C{H}$. An $\C{H}$-ideal $P$ is called an \emph{$\C{H}$-prime ideal} if, for any $\C{H}$-ideals $I$ and $J$, $IJ\subseteq P$ implies either $I\subseteq P$ or $J\subseteq P$. It can be shown that an $\C{H}$-prime ideal is, in fact, a prime ideal. The collection of all $\C{H}$-prime ideals is denoted by $\hspec$.

\begin{thm}[Goodearl and Letzter \cite{bg}]\label{Hstratification}
For the algebra $\C{A}=\Oq$, the following hold.
\begin{enumerate}
\item There are only finitely many $\C{H}$-prime ideals. 
\item
The set {\rm spec($\C{A}$)} can be partitioned into a disjoint union as follows:

$$\emph{spec}(\C{A}) = \bigcup_{J\in \mathcal{H}\text{\emph{-spec}}(\C{A})} Y_J,$$

where $$\displaystyle{Y_J:=\{P\in\text{\emph{spec}}(\mathcal{A}) \mid \bigcap_{h\in\mathcal{H}} h\cdot P = J\}}$$ is the \emph{$\C{H}$-stratum} associated to $J$. 
\item Each $\C{H}$-stratum is homeomorphic to the prime spectrum of a commutative Laurent polynomial ring over $\mathbb{K}$.
\item The primitive ideals of ${\rm spec}(\C{A})$ are precisely those ideals that are maximal within their $\C{H}$-stratum.
\end{enumerate}
\end{thm}

The \emph{dimension} of an $\C{H}$-stratum $Y_J$ is the Krull dimension of the commutative Laurent polynomial ring for which $Y_J$ is homeomorphic to by Theorem~\ref{Hstratification}(3). In other words, the dimension of $Y_J$ is the length $d$ of the longest chain $P_0\subset P_1\subset\cdots\subset P_d$ of prime ideals contained in $Y_J$.

\subsection{Cauchon Diagrams}
The deleting-derivations algorithm due to Cauchon and applied to the algebra $\Oq$ allows one to obtain a nice combinatorial parametrization of $\hspec$. 

\begin{defn}
An \emph{$m\times n$ diagram} is simply an $m\times n$ grid of squares, each square coloured either white or black. A diagram is a \emph{Cauchon diagram} if the colouring has the property that if a square is black, then either every square strictly above or every square strictly to the left is also black.
\end{defn}
\begin{figure}[htbp]
\begin{center}
\includegraphics[width=3.5in]{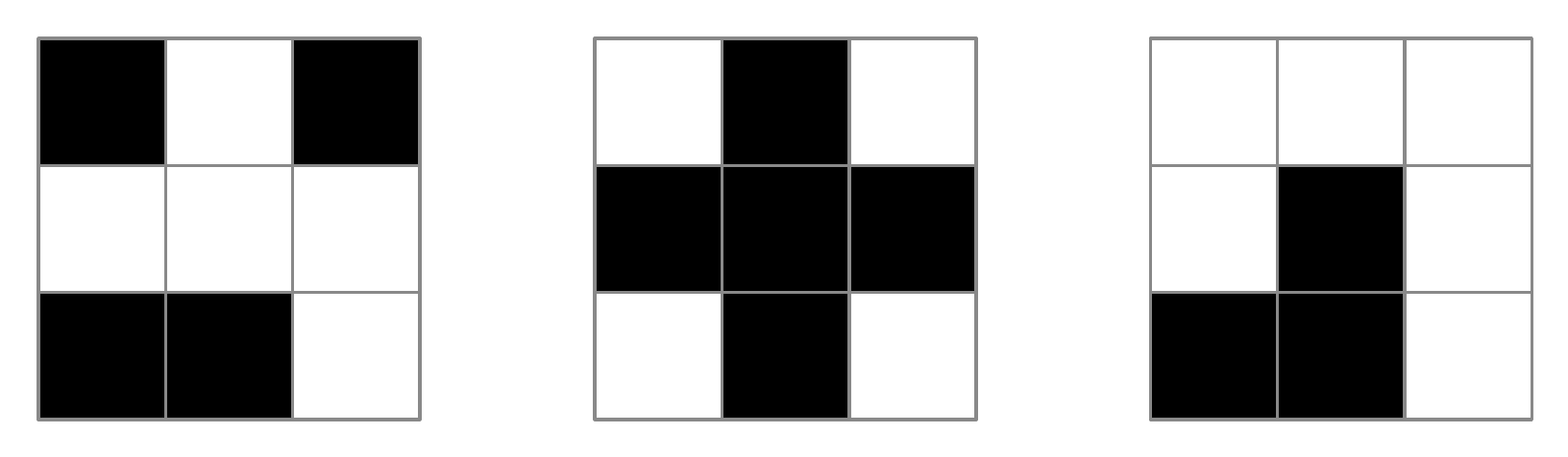}
\caption{Three $3\times 3$ diagrams. The left and center diagrams are Cauchon; the right diagram is not.}
\label{cauchonexamples}
\end{center}
\end{figure}

\begin{thm}[Cauchon~\cite{cauchon2}] \label{cauchonsthm}
For every $m$ and $n$, $\hspec$ is in bijective correspondence with the collection of $m\times n$ Cauchon diagrams.
\end{thm}

Given an $m\times n$ Cauchon diagram $D$ with $N$ white squares, it is convenient to label the white squares by the elements of $[N]$. Given such a labelling, we construct an $N\times N$ skew symmetric matrix $M(D)$ by the rule
\begin{displaymath}
 M(D)[i,j]= \left\{\begin{array}{ll}
1 & \textnormal{if square $i$ is strictly below or strictly to the right of square $j$,} \\
-1 & \textnormal{if square $i$ is strictly above or strictly to the left of square $j$, and}\\
0 & \textnormal{otherwise.}
\end{array} \right.
\end{displaymath}

The analysis of $M(D)$ has proven crucial in the study of the dimension of $\C{H}$-strata due to the following theorem.

\begin{thm}[\cite{bldim}] \label{BL}
Let $J$ be an $\C{H}$-prime ideal with corresponding Cauchon diagram $D$. The dimension of the $\C{H}$-stratum containing $J$ is $\dim(\ker(M(D))$.
\end{thm}

\section{Pipe Dreams and Permutations} \label{pipedreams}

In this section, we describe the \emph{pipe dreams} construction, which gives a bijection between Cauchon diagrams and restricted permutations.

Let us call a permutation $\sigma$ of $[m+n]$ \emph{restricted} if, for all $i\in[m+n]$, we have $-n\leq \sigma(i)-i\leq m$. The set of all restricted permutations of $[m+n]$ is a subposet of the symmetric group of $[m+n]$ endowed with the Bruhat order \cite{launois1}. Moreover we have the following result that was proved in \cite{launois1}.

\begin{thm}\label{orderisomorphic}
For fixed $m$ and $n$, the poset $\hspec$, ordered by inclusion, is order-isormophic to the poset of restricted permutations of $[m+n]$, ordered by the Bruhat order.
\end{thm}

The connection between Cauchon diagrams and restricted permutations suggested by Theorems \ref{cauchonsthm} and \ref{orderisomorphic} can be made clear via a notion called \emph{pipe dreams}. To explain this idea, let us fix an $m\times n$ diagram $D$. 
We lay ``pipes'' on the squares of $D$ by placing a ``hyperbola'' on every white square and a ``cross'' on every black square (see Figure~\ref{standard}). Next, we label the sides of the diagram by elements of $[m+n]$ as in, for example, Figure~\ref{standard} (we assume the general form of the labelling is obvious from this consideration). 
\begin{figure}[htbp]
\begin{center}
\includegraphics[width=2.5in]{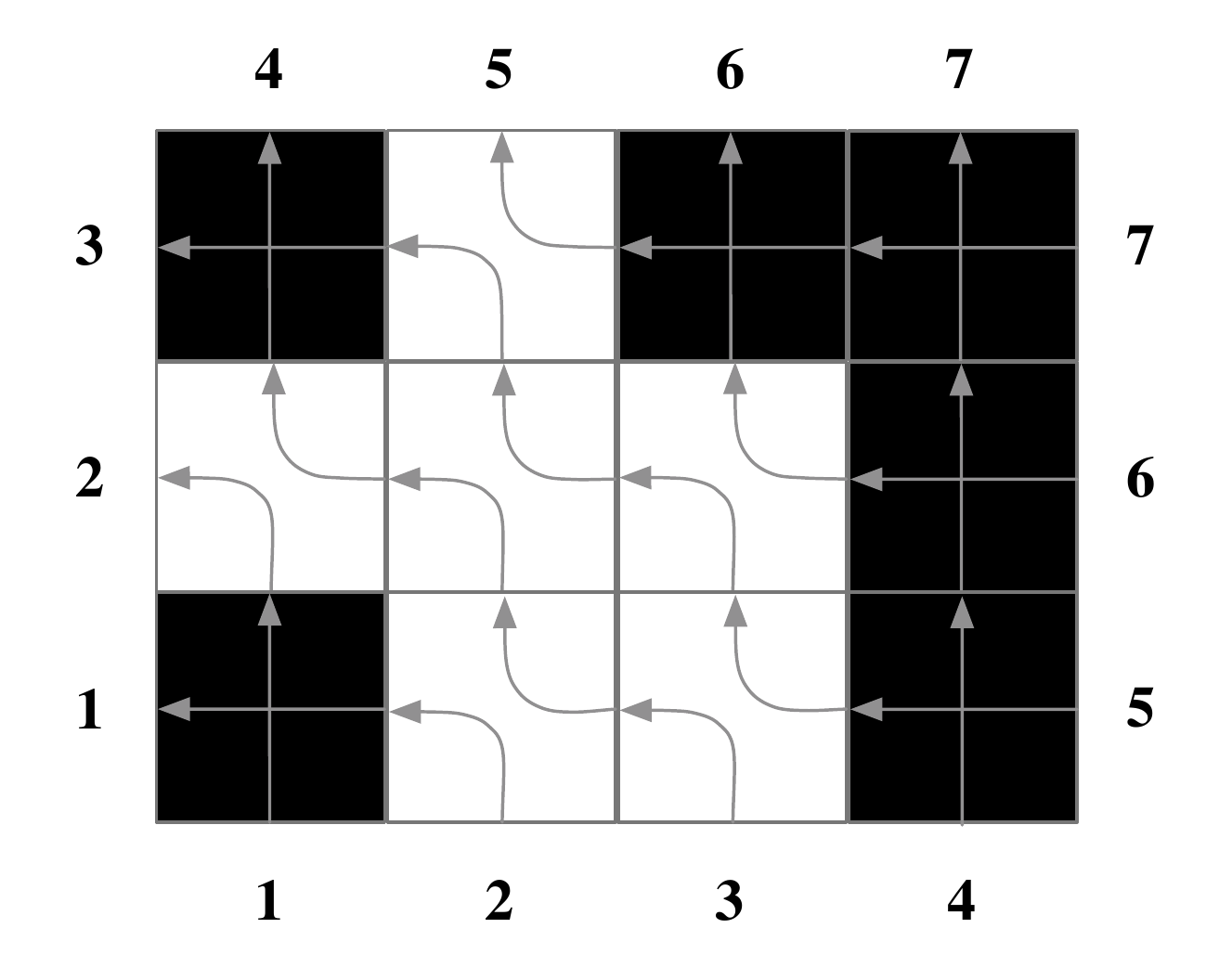}
\caption{Example of applying pipe dreams to a given diagram}
\label{standard}
\end{center}
\end{figure}
 
The restricted permutation $\sigma$ is obtained from this process by defining $\sigma(i)$ to be the label (on the left or top side of $D$) reached by following the pipe starting at label $i$ (on the bottom or right side of the $D$). To be clear, when following a pipe through a black square, we always go straight through the square.

For example, in the diagram of Figure~\ref{standard}, we obtain the restricted permutation $(1\,2)(\,3\,4\,7\,5)$. Notice that the inverse of the permutation is obtained by reversing this procedure, that is, starting at a label from the left or top and following the pipe to the bottom or right of the diagram.

\begin{defn} \label{omega}
We denote by $\omega$ the restricted $(m+n)$-permutation obtained from the $m\times n$ diagram consisting of only black squares. Therefore, 
\begin{displaymath}
\omega(i) = \left\{\begin{array}{ll}
m+i & \textnormal{if $1\leq i\leq n$,} \\
i-n & \textnormal{if $n+1\leq i\leq n+m$.}
\end{array} \right. 
\end{displaymath}The permutation $\omega$ is the maximum element in the set of restricted $(m+n)$-permutations ordered by the reverse Bruhat order. 
\end{defn}

\section{Dimension of the $\C{H}$-strata}
In this section we prove Theorem \ref{maintheorem}.
\subsection{Toric Permutations}
We introduce the notion of \emph{toric permutations}, which are essential in our analysis.
\begin{defn} \label{torusdef}
Given a diagram $D$ with restricted permutation $\sigma$, the \emph{toric permutation} associated to $D$ is simply the permutation $\tau=\sigma\omega^{-1}$.
\end{defn}

The toric permutation can be obtained via pipe dreams simply by following the same procedure except the bottom and right sides of the diagram are now relabelled as in Figure~\ref{pipes}. Notice that under this labelling, any cycle in the toric permutation may be found by considering the diagram as a ``torus'' and following the pipe around the ``torus''.
\begin{figure}[htbp]
\begin{center}
\includegraphics[width=2.5in]{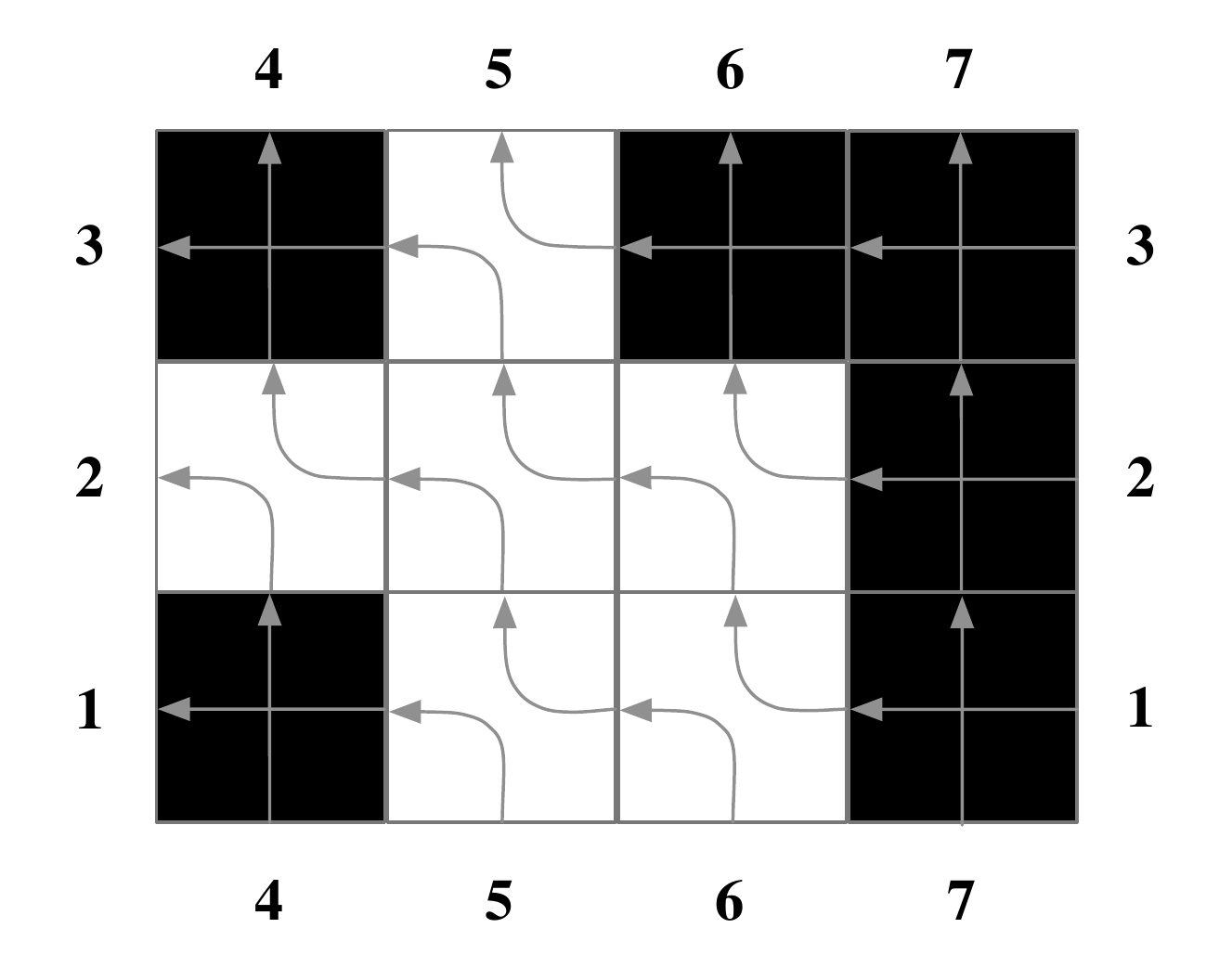}
\caption{New row/column labelling to obtain the toric permutation $(1\,3\,5)(2\,6\,4)(7)$}
\label{pipes}
\end{center}
\end{figure}

The proof of Theorem~\ref{maintheorem} requires the following two lemmas. For a permutation $\mu\in S_k$, $P_\mu$ denotes the corresponding $k\times k$ permutation matrix. That is, $P_\mu$ is the  $k\times k$ matrix whose entries are defined by $P_\mu[i,j]:= \delta_{j,\sigma(i)}$, where $\delta$ denotes the Kronecker symbol. 

\begin{lem} \label{technicallemma2}
Let $D$ be an $m\times n$ diagram whose corresponding restricted permutation is $\sigma$. If $\tau=\sigma\omega^{-1}$ is the toric permutation associated to $D$, then $\vb{v}\in\ker(P_{\omega}+P_{\sigma})$ if and only if $\vb{v}_{b}=-\vb{v}_{\tau(b)}$ for every label $b$, where $\vb{v}_i$ denotes the $i$th coordinate of $\vb{v}$.
\end{lem}

\begin{proof}
Consider the pipe in $D$ corresponding to the cycle in $\tau$ containing $b$ and $\tau(b)$. Now $\vb{v}\in\ker(P_{\omega}+P_{\sigma})$ if and only if for all $a$ we have $\vb{v}_{\omega(a)}+\vb{v}_{\sigma(a)} = 0$. Taking $a=\omega^{-1}(b)$ we obtain $\vb{v}_{b}=-\vb{v}_{\tau(b)}$ as desired.
\end{proof}
\begin{lem} \label{kernel}
Let $D$ be an $m\times n$ diagram whose corresponding restricted permutation is $\sigma$, and let $\tau$ denote the toric permutation $\tau=\sigma\omega^{-1}$.  Then the dimension of $\ker(P_\omega+P_{\sigma})$ is the number of odd cycles in the disjoint-cycle decomposition of $\tau$. 
\end{lem}

\begin{proof}
Let $\{\gamma_1,\ldots,\gamma_\ell\}$ be the set of odd cycles in the disjoint cycle decomposition of $\sigma\omega^{-1}$.
Given the odd cycle $\gamma=(a_1\, a_2\,\ldots\, a_{2k})$, define the vector $\vb{v}^\gamma \in \mathbb{Z}^{m+n}$ by
\begin{displaymath}
\vb{v}^\gamma_b = \left\{\begin{array}{ll}
1 & \textnormal{if $b=a_i$ and $i$ is odd,} \\
-1 & \textnormal{if $b=a_i$ and $i$ is even,}\\
0 & \textnormal{otherwise.}
\end{array} \right.
\end{displaymath}
We claim that the set $B=\{\vb{v}^{\gamma_i} \mid i\in [\ell]\}$ forms a basis for $\ker(P_{\omega}+P_{\sigma})$. Since the odd cycles are mutually disjoint, it is clear that the members of $B$ form an independent set in the $\mathbb{Z}$-module $\mathbb{Z}^{m+n}$. 

Suppose that $\vb{v}\in \ker(P_{\omega}+P_{\sigma})$.  By Lemma~\ref{technicallemma2}, we know that if $b$ is in an even cycle of $\tau$, then $\vb{v}_b=0$. Moreover, the values of the entries corresponding to an odd cycle agree up to multiplication by $-1$. Thus we see that $\vb{v}$ can be written as a linear combination of elements of $B$.
\end{proof}

\subsection{Proof of Theorem~\ref{maintheorem}}
\begin{notn}
Fix an $m\times n$ diagram $D$ with $N$ white squares labelled by distinct elements of the set $[N]$ such that labels are strictly increasing from left to right along rows and if $i < j$ then the label of each white box in row $i$ is strictly less than the label of each white box in row $j$. Let $\tau=\sigma\omega^{-1}$ be the toric permutation of $D$. Let $\vb{w}$ be in the column space of $M(D)$ and $\vb{v}$ in the column space of $P_\omega+P_\sigma$. We refer to the entries of $\vb{w}$ by $\vb{w}_j$ for $j\in [N]$ and the entries of $\vb{v}$ by $\vb{v}_a$ where $a\in [m+n]$.
\begin{enumerate}
\item Since, in the toric labelling, each side of a row or column is given the same label, we may unambiguously refer to a row or column by this label. 
\item Given a diagram $D$ with the toric labelling and a white square $i$ of $D$, let $\lef{i}$ and $\up{i}$ be, respectively, the labels of the rows or columns reached by following the bottom and top pipes of the hyperbola pipe placed on $i$. See Example \ref{tne}.
\item For $S\subseteq [N]$, let $\vb{w}_S = \sum_{j\in S}\vb{w}_j$. 
\item For a given white square $i$, let $A(i), R(i), B(i)$ and $L(i)$ be the sets of white squares that are, respectively, strictly above, strictly to the right, strictly below and strictly to the left of square $i$.
\end{enumerate}
\end{notn}

\begin{figure}[htbp]
\begin{center}
\includegraphics[width=2.5in]{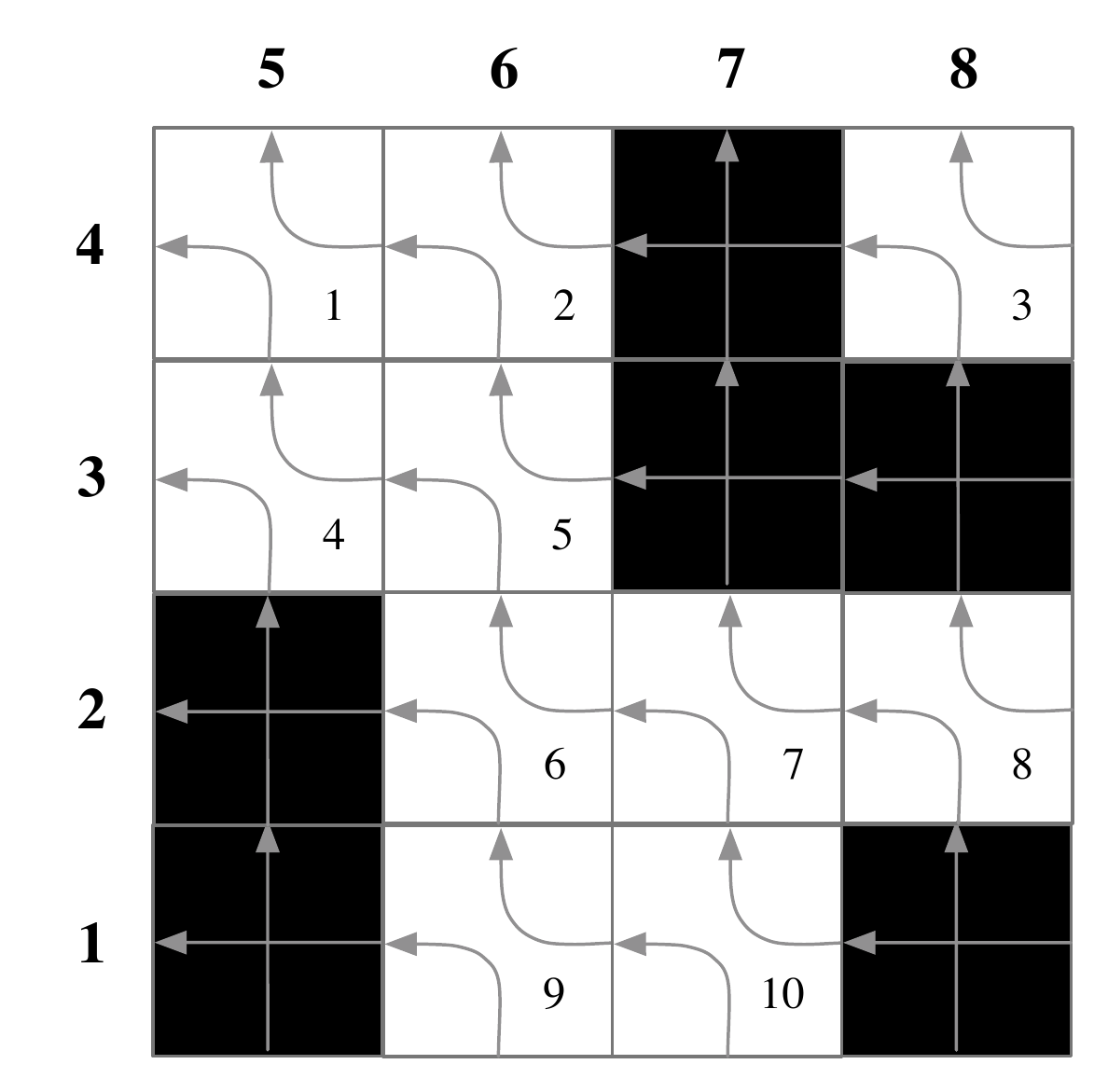}
\caption{}
\label{tnefig}
\end{center}
\end{figure}
\begin{ex}\label{tne}
Consider the diagram in Figure~\ref{tnefig}. For this example, we have labelled the white squares in a regular font, and the row and column labels in bold font. We have, for example, $\lef{7}=\vb{4}$ and $\up{7}=\vb{7}$, while $\lef{8}=\vb{7}$ and $\up{8}=\vb{6}$. 
On the other hand $A(5)=\{2\}, R(5)=\emptyset, B(5)=\{6,9\}$ and $L(5)=\{4\}$.
\end{ex}
Before proving the main theorem of this section, we note that by the definition of $M(D)$ we have $\vb{w}\in\ker(M(D))$ if and only if for every white square $i$, the following identity holds $$\vb{w}_{A(i)}+\vb{w}_{L(i)}=\vb{w}_{B(i)}+\vb{w}_{R(i)}.$$ 

\begin{thm} \label{main2}
If $D$ is a diagram and  $\sigma$ is the restricted permutation obtained from $D$, then $$\ker(P_\omega+P_\sigma) \simeq \ker(M(D)).$$

\end{thm}

\begin{proof}
We place the toric labelling on $D$. Suppose that $D$ has $N$ white squares labelled $1,2,\ldots, N$. The theorem is proved once we exhibit injective functions $\phi: \ker(P_\omega+P_\sigma) \rightarrow \ker(M(D))$ and $\psi: \ker(M(D)) \rightarrow \ker(P_\omega+P_\sigma)$. The reader may wish to refer to the example immediately following this proof where we give an explicit calculation of $\phi$ and $\psi$ for the diagram of Figure~\ref{tnefig}.

Given $\vb{v}\in\ker(P_\omega+P_\sigma)$, we take $\vb{w}:=\phi(\vb{v})$ to be the vector whose entries are defined by $\vb{w}_i = \vb{v}_{\lef{i}}-\vb{v}_{\up{i}}$. To show $\vb{w}\in\ker(M(D))$ we need only to check that for all white squares $i$ the relation $\vb{w}_{A(i)}+\vb{w}_{L(i)}-\vb{w}_{B(i)}-\vb{w}_{R(i)}=0$ holds. 

Notice that, if $S=\{s_1,s_{2},\ldots,s_k\}$ is any block of consecutive white squares in the same row, where $s_1$ is the left-most white square, and $s_k$ is the right-most, then we have $\vb{w}_S=\vb{v}_{\lef{s_1}}-\vb{v}_{\up{s_k}}$ since one can easily check that $\up{s_i}=\lef{s_{i+1}}$ for $i \in [k-1]$. For a small example, see Figure~\ref{proof1}.  \begin{figure}[htbp]
\begin{center}
\includegraphics[width=2.5in]{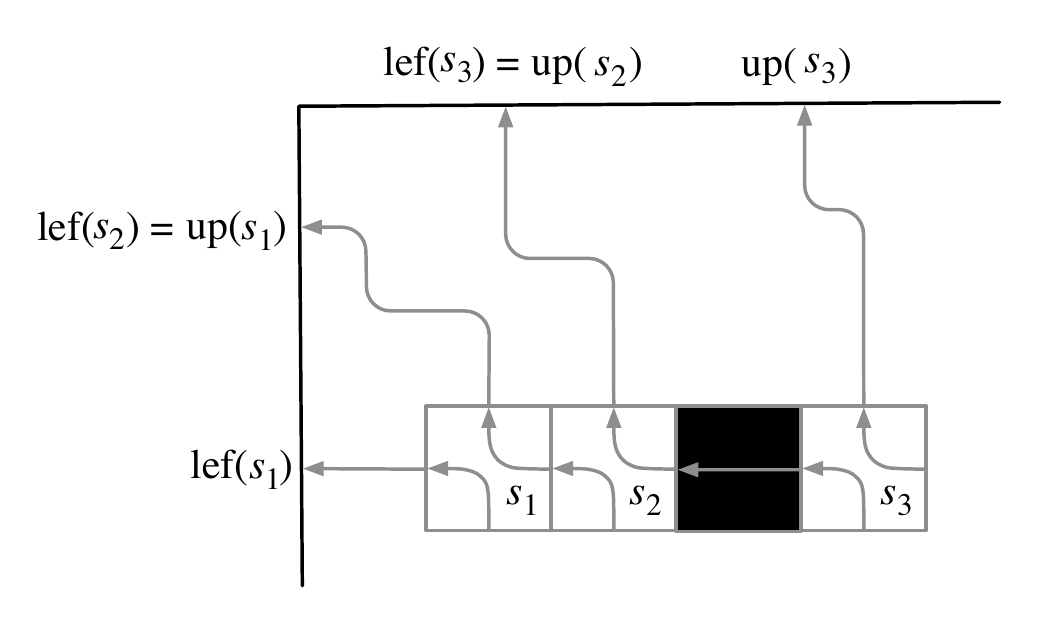}
\caption{$\vb{w}_{\{s_1,s_2,s_3\}} = \vb{v}_{\lef{s_1}} - \vb{v}_{\up{s_3}}$}
\label{proof1}
\end{center}
\end{figure}Similarly, if $S=\{s_1,s_{2},\ldots,s_k\}$ is any block of consecutive white squares in the same column, where $s_1$ is the bottom-most white square, and $s_k$ is the top-most white square, then $\vb{w}_S=\vb{v}_{\lef{s_1}}-\vb{v}_{\up{s_k}}$. 

Fix a white square $i$. For the sake of brevity, we assume that $A(i),L(i),B(i)$ and $R(i)$ are all non-empty; the remaining cases are similar. We now show that $\vb{w}_{A(i)}+\vb{w}_{L(i)}-\vb{w}_{B(i)}-\vb{w}_{R(i)} =0$ (Figure~\ref{proof2}). Suppose that the label of the row containing $i$ is $r$ and the label of the column containing $i$ is $c$. Let $s_{r,1}$ be the left-most white square in row $r$, $s_{r,k}$ the right-most white square in row $r$, $s_{c,1}$ the bottom-most white square in column $c$ and $s_{c,\ell}$ the top-most white square in column $c$. 
\begin{figure}
\begin{center}
\includegraphics[width=2.5in]{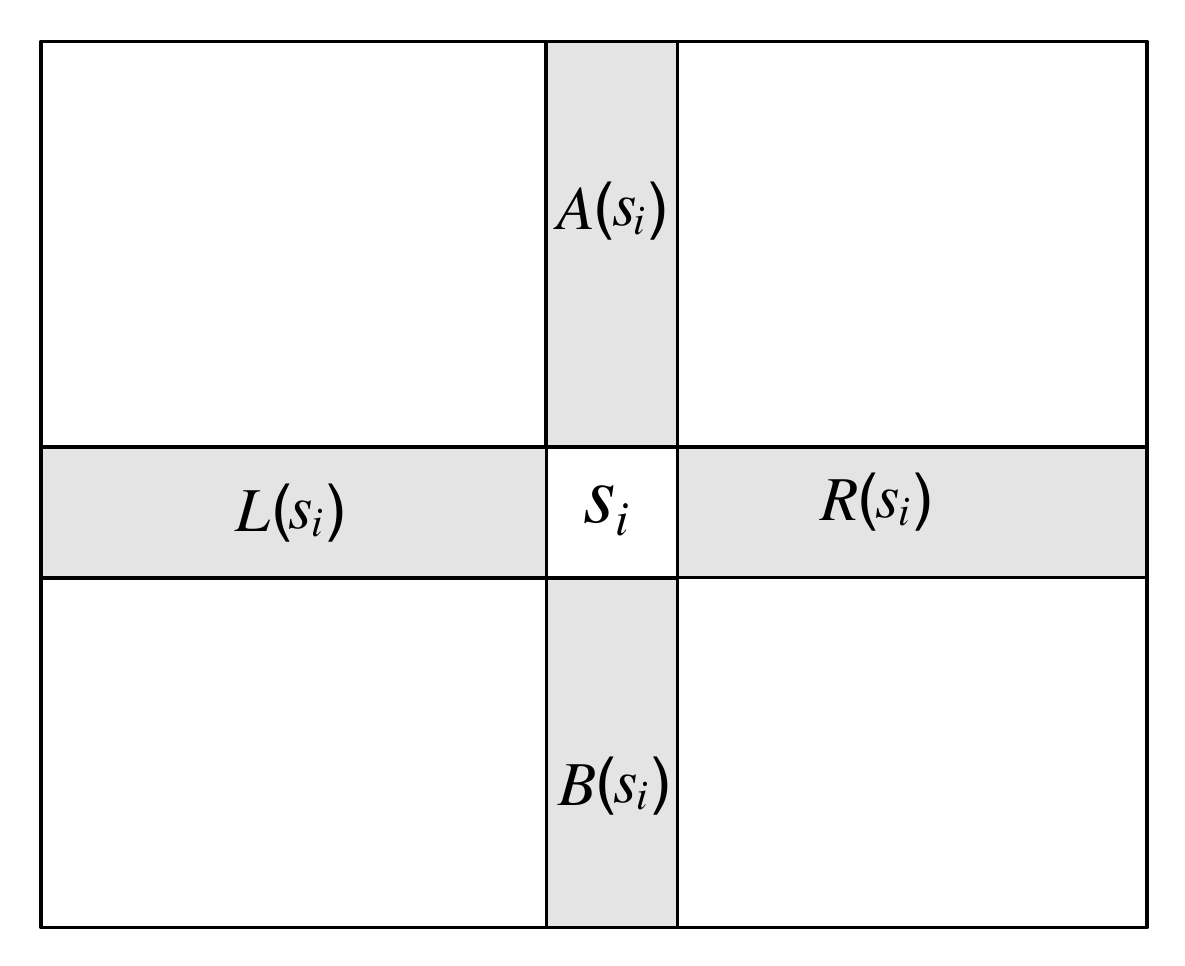}
\caption{}
\label{proof2}
\end{center}
\end{figure}We have 
\begin{eqnarray*}
\vb{w}_{A(i)}+\vb{w}_{L(i)}-\vb{w}_{B(i)}-\vb{w}_{R(i)} &= &  (\vb{v}_{\up{i}} - \vb{v}_{\up{s_{c,\ell}}}) + (\vb{v}_{\lef{s_{r,1}}} - \vb{v}_{\lef{i}}) \\
  && -(\vb{v}_{\lef{s_{c,1}}} - \vb{v}_{\lef{i}}) - (\vb{v}_{\up{i}} - \vb{v}_{\up{s_{r,k}}}) \\
& = & (\vb{v}_{\lef{s_{r,1}}} + \vb{v}_{\up{s_{r,k}}}) - (\vb{v}_{\up{s_{c,1}}}+\vb{v}_{\lef{s_{c,\ell}}})\\
&=& (\vb{v}_r + \vb{v}_{\tau(r)}) - (\vb{v}_c+\vb{v}_{\tau(c)})\\
& = & 0,
\end{eqnarray*}
where the last equality follows by Lemma~\ref{technicallemma2}.

Now define $\psi: \ker(M(D)) \rightarrow \ker(P_\omega+P_\sigma)$ as follows. Let $\vb{w}\in \ker(M(D))$ and write $\vb{v}=\psi(\vb{w})$. If $a$ is a column we take $\vb{v}_a$ to be the sum of all white squares in column $a$, but if $a$ is a row, then we take $\vb{v}_a$ to be the negative of the sum of all white squares in row $a$.  By Lemma~\ref{technicallemma2}, we want to show that for all labels $a$, we have
\begin{eqnarray}
\vb{v}_a+\vb{v}_{\tau(a)}& = &0.\label{one} 
\end{eqnarray}

Fix a label $a \in [m+n]$ and let $(s_1,s_2,\ldots,s_k)$ be the sequence of white squares used in the pipe from $a$ to $\tau(a)$. \begin{figure}[htbp]
\begin{center}
\includegraphics[width=3.5in]{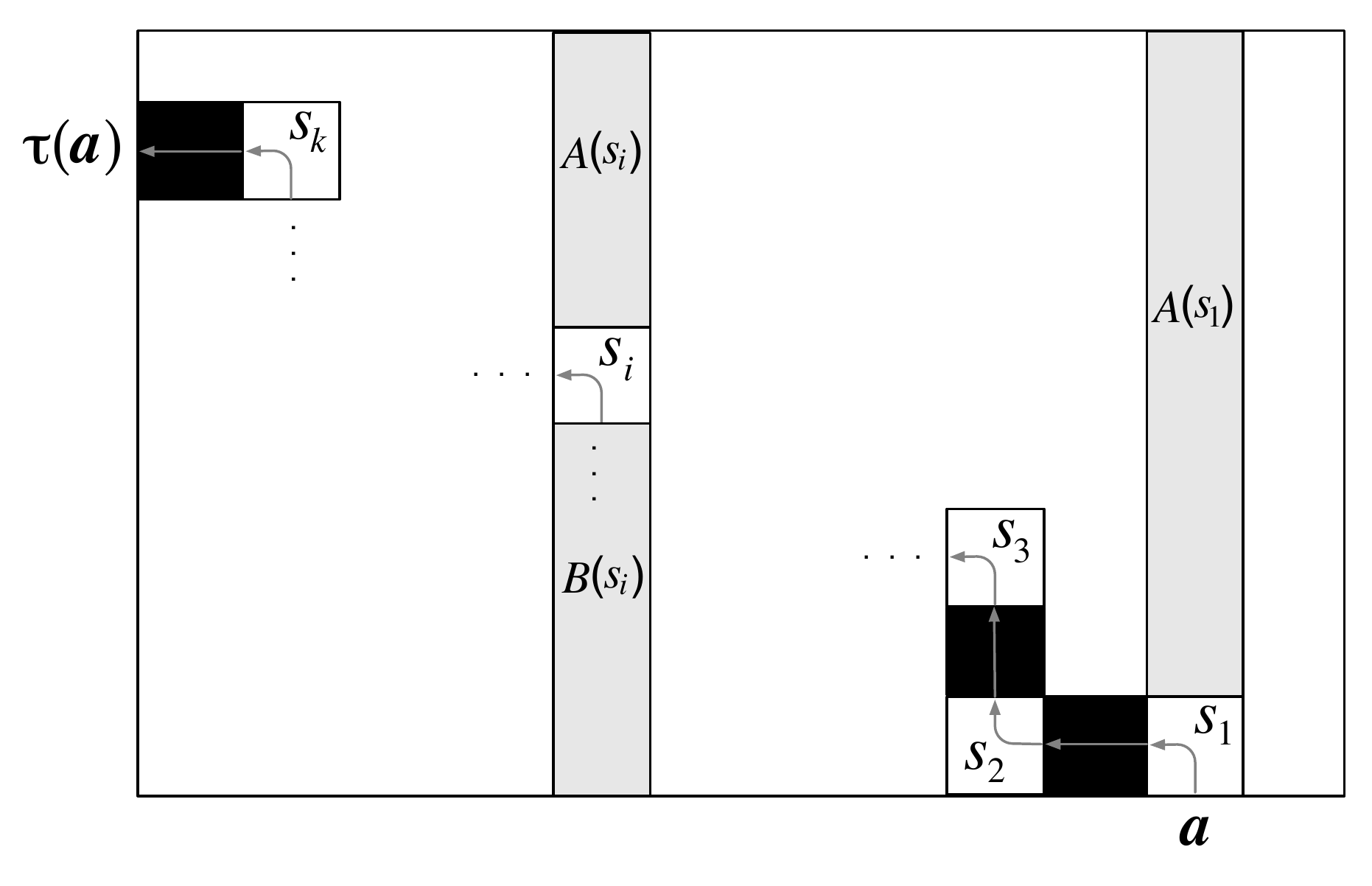}
\caption{}
\label{proof3}
\end{center}
\end{figure}We assume that $a$ is a column as the case in which $a$ is a row is argued similarly. We first prove by induction that, for every $i=1,\ldots, k$,  we have
\begin{eqnarray}
\vb{v}_a &= & \vb{w}_{s_1} + \vb{w}_{A(s_1)} \nonumber \\
&= &(-1)^{i+1}\vb{w}_{s_i}+\vb{w}_{A(s_i)}-\vb{w}_{B(s_i)}.\label{two1}
\end{eqnarray}

See Figure~\ref{proof3}. If $i=1$, then Equation~(\ref{two1}) is trivially true since $B(s_i)=\emptyset$. So suppose that it is true for all values less than $i>1$. There are two cases to consider. If $i$ is even, then $s_i$ is the first white square to the left of $s_{i-1}$. Note that $\vb{w}_{R(s_i)} = \vb{w}_{R(s_{i-1})} + \vb{w}_{s_{i-1}}$ and $\vb{w}_{L(s_{i-1})} = \vb{w}_{L(s_{i})} + \vb{w}_{s_{i}}$.  Since $\vb{w}\in\ker(M(D))$ we have the two equations

$$\vb{w}_{A(s_{i-1})} + \vb{w}_{L(s_{i})}+\vb{w}_{s_{i}} - \vb{w}_{B(s_{i-1})} - \vb{w}_{R(s_{i-1})}  =  0,$$ and
$$\vb{w}_{A(s_{i})} + \vb{w}_{L(s_{i})} - \vb{w}_{B(s_{i})} - \vb{w}_{R(s_{i-1})}-\vb{w}_{s_{i-1}}  =  0.$$

Subtracting the second equation from the first we obtain
\begin{eqnarray}
\vb{w}_{s_{i-1}}+\vb{w}_{A(s_{i-1})}-\vb{w}_{B(s_{i-1})} &=& -\vb{w}_{s_{i}} + \vb{w}_{A(s_{i})}-\vb{w}_{B(s_{i})}. \label{three1}
\end{eqnarray}
But since $i-1$ is odd, the left hand side of~(\ref{three1}) is equal, by induction, to $\vb{w}_{s_1} + \vb{w}_{A(s_1)}$ and thus~(\ref{three1}) implies~(\ref{two1}) for the case that $i$ is even. 

Now for $i$ odd, $s_{i}$ is the first white square above $s_{i-1}$. It is easy to check that we have the two equations 
$$\vb{w}_{A(s_{i-1})} = \vb{w}_{A(s_{i})} + \vb{w}_{s_{i}}$$ and 
$$\vb{w}_{B(s_{i})} = \vb{w}_{B(s_{i-1})} + \vb{w}_{s_{i-1}}.$$
By induction we obtain
\begin{eqnarray*}
\vb{w}_{s_1} + \vb{w}_{A(s_1)} &= &-\vb{w}_{s_{i-1}}+\vb{w}_{A(s_{i-1})}-\vb{w}_{B(s_{i-1})}\\
& = & -\vb{w}_{s_{i-1}} +\vb{w}_{A(s_{i})} + \vb{w}_{s_{i}} - (\vb{w}_{B(s_{i})} - \vb{w}_{s_{i-1}})\\
& = & \vb{w}_{s_{i}}  + \vb{w}_{A(s_{i})} - \vb{w}_{B(s_{i})}.
\end{eqnarray*}
 This finishes the proof of Equation~(\ref{two1}).
 
Now we verify Equation~(\ref{one}). If $k$ is even, then $\tau(a)$ is a column-label and so $A(s_k)=\emptyset$. By~(\ref{two1}),
\begin{eqnarray*}
\vb{v}_a  & = & \vb{w}_{s_1} + \vb{w}_{A(s_1)}\\
& = & -\vb{w}_{s_k}+\vb{w}_{A(s_k)}-\vb{w}_{B(s_k)}\\
&=& -\vb{w}_{s_k}-\vb{w}_{B(s_k)}\\
&=& -\vb{v}_{\tau(a)}.
\end{eqnarray*}
On the other hand, if $k$ is odd, then $\tau(a)$ is a row-label and so $L(s_k)=\emptyset$. Since $\vb{w}\in\ker(M(D))$, we must have $\vb{w}_{A(s_k)}-\vb{w}_{B(s_k)}=\vb{w}_{R(s_k)}$.
Hence
\begin{eqnarray*}
\vb{v}_a  & = & \vb{w}_{s_1} + \vb{w}_{A(s_1)}\\
& = & \vb{w}_{s_k}+\vb{w}_{A(s_k)}-\vb{w}_{B(s_k)}\\
&=& \vb{w}_{s_k}+\vb{w}_{R(s_k)}\\
&=& -\vb{v}_{\tau(a)}.
\end{eqnarray*}

The final step is to prove that the functions $\phi$ and $\psi$ are injective. First, suppose that $\psi(\vb{w})=\vb{0}$ with $\vb{w}\neq\vb{0}$. There must, therefore, exist a white square $i$ (with $i$ minimum) such that $\vb{w}_i\neq 0$ but $\vb{w}_{A(i)}=0$ and $\vb{w}_{L(i)}=0$. Since $\vb{w}\in\ker(M(D))$ we have $\vb{w}_{B(i)}+\vb{w}_{R(i)}=0$. On the other hand, since $\psi(\vb{w})=\vb{0}$ we have, by the construction of $\psi(\vb{w})$, that $\vb{w}_{B(i)}+\vb{w}_{i}=0$ and $0=\vb{w}_{R(i)}+\vb{w}_{i}=-\vb{w}_{B(i)}+\vb{w}_{i}$. Thus we must have $\vb{w}_i=0$, contradicting the choice of $i$. Therefore $\psi$ is injective.

To show that $\phi$ is injective we show that $\psi(\phi(\vb{v}))=-2\vb{v}$ for every $\vb{v}\in\ker(P_\omega+P_\sigma)$. First, if $a$ is a row then 
$$\psi(\phi(\vb{v}))_a = -\sum_{i \mbox{ in column }a} \phi(\vb{v})_i= -\sum_{i \mbox{ in column }a} (\vb{v}_{\lef{i}}-\vb{v}_{\up{i}})=-(\vb{v}_a - \vb{v}_{\tau(a)}).$$ On the other hand, by Lemma~\ref{technicallemma2}, we know that $\vb{v}_{\tau(a)} = -\vb{v}_a$ and thus $\psi(\phi(\vb{v}))_a = -2\vb{v}_a$. A similar argument shows that if $a$ is a column then $\psi(\phi(\vb{v}))_a = -2\vb{v}_a$. Hence $\phi$ is injective.
\end{proof}

\begin{rem}
The linear maps $\phi$ and $\psi$ defined in the proof of Theorem \ref{main2} are not inverses. However one can check that 
$$\psi\circ \phi = -2\cdot {\rm id}|_{{\rm ker}(P_\omega+P_\sigma)} \mbox{ and }
\phi\circ \psi=-2\cdot {\rm id}|_{{\rm ker}(M(D))}.$$
\end{rem}

\begin{ex}
In this example we give examples of the maps $\phi$ and $\psi$ from the above proof. If $M$ is a matrix, then $M^t$ denotes its transpose. We consider the diagram from Figure~\ref{tnefig}. The corresponding toric permutation is $(1\,4\,8\,7\,2\,6)(3\,5)$.  By Lemma~\ref{kernel}, we may construct $\vb{v}=(1,1,0,-1,0,-1,-1,1)^t$ corresponding to the cycle $(1\,4\,8\,7\,2\,6)$, and $\vb{v'}:=(0,0,1,0,-1,0,0,0)^t$ corresponding to the cycle $(3\,5)$. 
\begin{figure}[htbp]
\begin{center}
\includegraphics[width=2.5in]{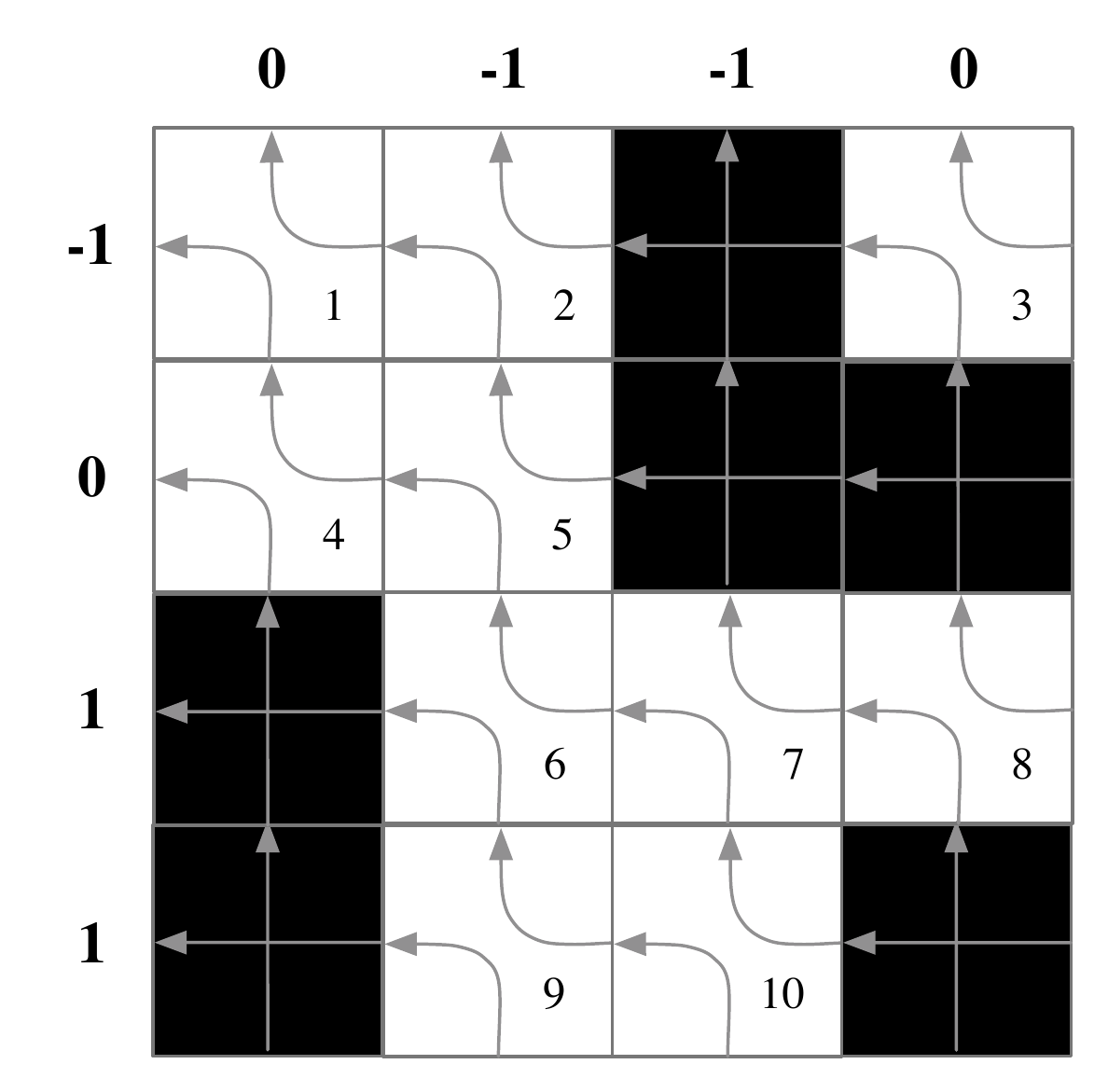}
\caption{Calculating $\phi((1,1,0,-1,0,-1,-1,1)^t)$ }
\label{tnefig2}
\end{center}
\end{figure}
In Figure~\ref{tnefig2}, we now replaced the row or column-label $i$ with $\vb{v}_i$. We wish to construct an element $\phi(\vb{v})=\vb{w}$ in $\ker((M(D))$. By the isomorphism in the proof of Theorem~\ref{main2}, we see that, for example, $\vb{w}_8= \vb{v}_{\lef{8}}-\vb{v}_{\up{8}} = -1-(-1)= 0$. Continuing in this way, we find that $\vb{w}=(-1,1,-2,1,-1,2,0,0,0,2)^t$. 

In Figure~\ref{tnefig3}, we have replaced the white square labels from Figure~\ref{tnefig} with the corresponding value of $\vb{w}$. Using this, let us calculate $\psi(\vb{w})=\vb{u}$. This is easier: if $i$ is a row, then $\vb{u}_i$ is simply the negative of the row sum, and if $i$ is a column, then $\vb{u}_i$ is the column sum. Thus we obtain, $\vb{u}=(-2,-2,0,2,0,2,2,-2)^t$. It is easy to check that $\vb{u}=-2\vb{v} \in \ker(P_\sigma+P_\omega)$. 
\begin{figure}[htbp]
\begin{center}
\includegraphics[width=2.5in]{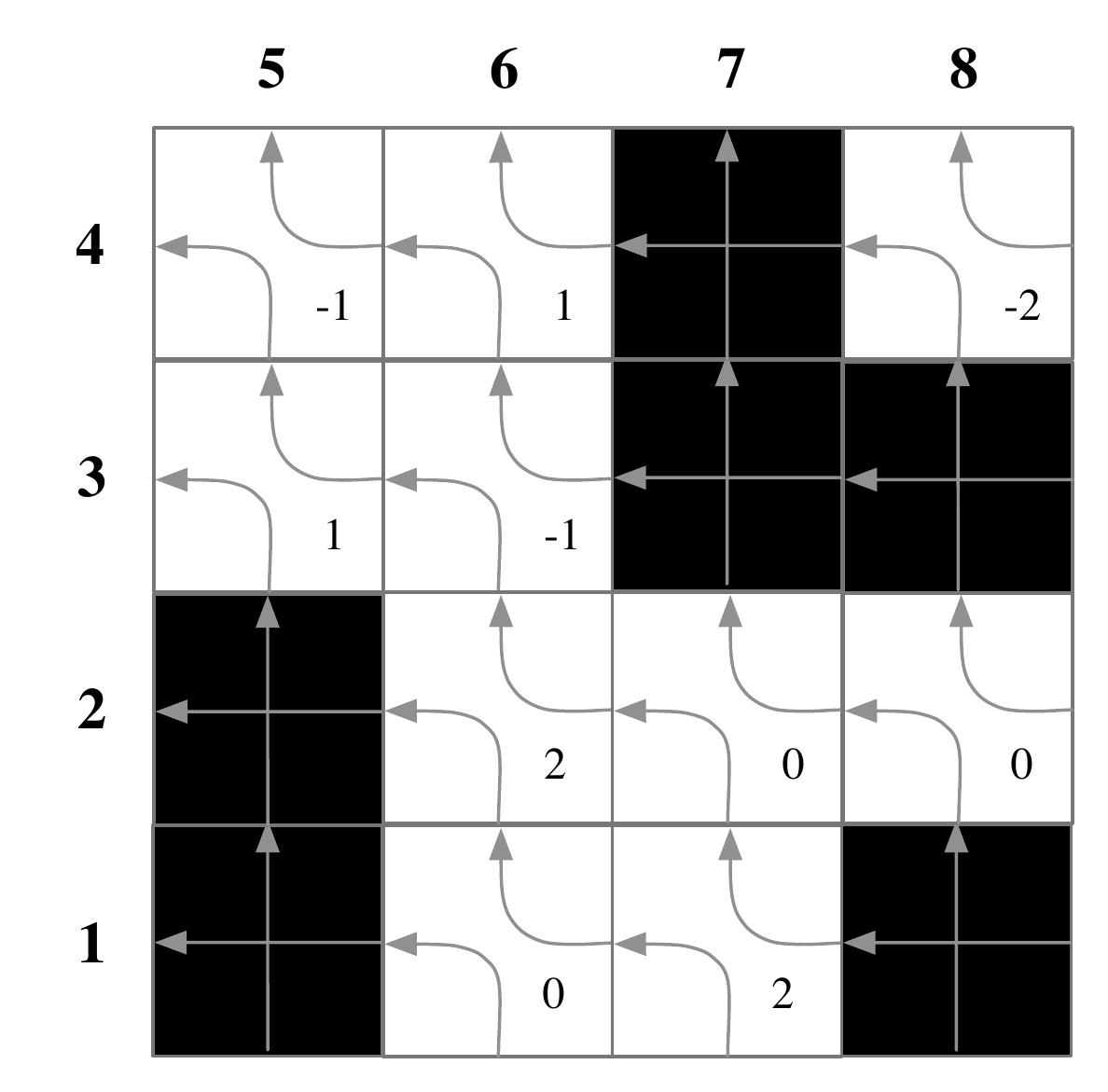}
\caption{Calculating $\psi((-1,1,-2,1,-1,2,0,0,0,2)^t).$ }
\label{tnefig3}
\end{center}
\end{figure}
\end{ex}

\begin{cor} \label{maincor2}
For a Cauchon diagram $D$, the dimension of the $\C{H}$-stratum corresponding to $D$ is equal to the number of odd cycles in the disjoint cycle decomposition of the toric permutation associated to $D$.
\end{cor}

\begin{proof}
This follows immediately from Theorems~\ref{BL} and~\ref{main2}.
\end{proof}

\section{Enumeration of $\C{H}$-primes with respect to dimension}
In this section we apply Corollary~\ref{maincor2} to obtain enumeration formulas for the total number of $m\times n$ $\mathcal{H}$-strata of a given dimension. 

Suppose that we are given an $m\times n$ diagram $D$ with the property that the disjoint cycle decomposition of the toric permutation $\tau$ consists of exactly one cycle $\tau=(a_1\,a_2\,\ldots\,a_{m+n})$. Without loss of generality, take $a_1=1$. Using the pipe-dreams visualization of $\tau$, it is easy to check that the sequence $(a_1,a_2,\ldots,a_{m+n})$ consists of contigious subsequences $R_1,C_1,\ldots, R_k, C_k$ alternating between increasing sets of row-labels (the $R_i)$ and decreasing sets of column-labels (the $C_i$). In other words, we may write $\tau=(R_1, C_1, R_2,\ldots, R_k, C_k)$ where the $R_i$ form a partition of $[m]$ and each $R_i$ is written in increasing order, while the $C_i$ form a partition of $m+[n]$ and each $C_i$ is written in decreasing order. We call $\tau$ a \emph{toric cycle}.

On the other hand, given a toric cycle $\tau=(R_1, C_1, R_2,\ldots, R_k, C_k)$, it is a simple matter to check that the permutation $\tau\omega$ is restricted in the sense of Section~\ref{pipedreams}. By Theorems~\ref{cauchonsthm} and~\ref{orderisomorphic}, we conclude that there exists a unique Cauchon diagram whose toric permutation is $\tau$. 

\begin{defn}
The number of partitions of $[n]$ into $k$ non-empty parts is the \emph{Stirling number of the second kind}, and will here be denoted by ${n\brace k}$. Note that some authors write $S(n,k)$ for ${n\brace k}$. 
\end{defn}

The above discussion implies that the number $d_{m,n}$  of $m\times n$ diagrams whose toric permutation consists of exactly one cycle is exactly $$d_{m,n} = \sum_{k=1}^{\min(m,n)} k!(k-1)!{m\brace k}{n\brace k}.$$ 

Let $\mathcal{D}(x,y):=\sum_{m,n\geq 1}d_{m,n}\frac{x^m}{m!}\frac{y^n}{n!}$ be the exponential generating function of $d_{m,n}$ and let $C(x,y)$ be the exponential generating function of all toric permutations (and hence of Cauchon diagrams). The relation between these two generating functions is provided to us via the well-known exponential formula. Note that while this formula is given in Section 5.1 of Stanley~\cite{stanley} (see also Chapter 3 in Wilf~\cite{wilf}) for single variable exponential generating functions,  it is straightforward to generalize the result to multivariable exponential generating functions. 

For those readers unfamiliar with the exponential formula, we now roughly describe the approach. Suppose we wish to find the exponential generating function $A(x,y)$ of a set of a structures, where each structure can be written uniquely as a disjoint union of labelled ``irreducible'' components. For example, in the case that is of interest to us, the ``structures'' are the toric permutations, and the labelled irreducible components are the toric cycles. The exponential formula says that if $B(x,y)$ is the exponential generating function for the labelled irreducible components, then $A(x,y)=\exp(B(x,y))$. In our case then, we obtain the following formula. 
\begin{eqnarray}
C(x,y)&=&\exp(x+y)\exp(\mathcal{D}(x,y))\label{eqnC1},
\end{eqnarray} where the extra $\exp(x+y)$ is the exponential generating function for the all-black diagrams (i.e., the identity toric permutation).

We may refine $\mathcal{D}(x,y)$ by noting that $$\mathcal{D}_e(x,y)=\frac{1}{2}(\mathcal{D}(x,y) - \mathcal{D}(-x,-y))$$ is the generating function for the even toric cycles, while $$\mathcal{D}_o(x,y)=\frac{1}{2}(\mathcal{D}(x,y) + \mathcal{D}(-x,-y))$$ is the generating function for the odd toric cycles. 

Therefore, if $C(x,y,t)$ is the generating function whose coefficient of $\frac{x^n}{n!}\frac{y^m}{m!}t^d$ is the number of $m\times n$ Cauchon diagrams with $d$ odd cycles in the toric permutation, then 
\begin{eqnarray} 
C(x,y,t) & = & \exp(x+y+\mathcal{D}_e(x,y) +t\mathcal{D}_o(x,y)) \nonumber \\
& = & \exp(x+y)\exp(\mathcal{D}(x,y))^{\frac{t+1}{2}}\exp(\mathcal{D}(-x,-y))^{\frac{t-1}{2}}.\label{eqnD}
\end{eqnarray}
On the other hand, it follows from \cite[Corollary 1.5]{launois1} that the number of $\C{H}$-primes in $\Oq$ is the poly-Bernoulli number $B_n^{(-m)}$ \cite{kaneko2}.  As a consequence, we may apply the work of Kaneko~\cite[Remark~p~223]{kaneko2} to conclude that $C(x,y)$ satisfies
\begin{eqnarray} 
C(x,y)&=&\frac{e^{x+y}}{e^x+e^y-e^{x+y}}.\label{eqnC2}
\end{eqnarray}
We thus obtain
\begin{thm}\label{egf4}
If $h(m,n,d)$ is the number of $m\times n$ Cauchon diagrams whose toric permutation has $d$ odd cycles, then $C(x,y,t) = \sum_{m,n,d}  h(m,n,d)t^d\frac{x^m}{m!}\frac{y^n}{n!}$ satisfies 
\begin{eqnarray}
C(x,y,t) &= & (e^{-y}+e^{-x}-1)^\frac{-1-t}{2}(e^x+e^y-1)^\frac{1-t}{2}.\label{eqnD2}
\end{eqnarray}

\end{thm}

\begin{proof}
Comparing Equations~(\ref{eqnC1}) and~(\ref{eqnC2}) we see that 
\begin{eqnarray*} 
\exp(D(x,y)) &= &(e^x + e^y-e^{x+y})^{-1}\\
& = & e^{-x-y}(e^{-y}+e^{-x}-1)^{-1}.
\end{eqnarray*}
Substituting the latter equality into Equation~(\ref{eqnD}) leads to Equation~(\ref{eqnD2}), as desired.
\end{proof}
Corollary~\ref{maincor2} and Theorem~\ref{egf4} immediately give us
\begin{cor} \label{enumerationcor}
If $h(m,n,d)$ is the number of $d$-dimensional $\C{H}$-strata in the prime spectrum of $\Oq$, then $$C(x,y,t) = \sum_{m,n,d} h(m,n,d)t^d\frac{x^m}{m!}\frac{y^n}{n!}$$ satisfies 
\begin{eqnarray}
C(x,y,t) &= & (e^{-y}+e^{-x}-1)^\frac{-1-t}{2}(e^x+e^y-1)^\frac{1-t}{2}.
\end{eqnarray}\qed
\end{cor}

Bell and Launois, together with Nguyen~\cite{bln} and Lutley~\cite{bll}, have given exact formulas for the number of $2\times n$ and $3\times n$ primitive $\C{H}$-strata respectively. Using these formulas, asymptotic results are obtained and a general conjecture is proposed in~\cite{bldim} for the asymptotic proportion of $d$-dimensional $m\times n$ $\C{H}$-strata for fixed $d$ and $m$. While it turns out that the conjecture is false in general, we are now in a position to give the correct asymptotic proportions (see Theorem~\ref{egf2cor}). The most important result towards this goal is given in the following theorem.

\begin{thm} \label{egf2}
If $m,n >0$ and $d\geq 0$ are integers, then for all integers $k\in\{1-m,\ldots ,m+1\}$ there exist rational numbers $c_k(m,d)$ such that the number $h(m,n,d)$ of $m\times n$ $d$-dimensional $\mathcal{H}$-strata in the prime spectrum of $\Oq$ is given by $$h(m,n,d) = \sum_{k=1-m}^{m+1} c_k(m,d)k^n.$$

Moreover, if $a(d)=[t^d](t+1)(t+3)\cdots(t+2m-1)$, then $$c_{m+1}(m,d) = 2^{-m}a(d),$$ 
where $[t^d]p(t)$ denotes the coefficient of $t^d$ in the polynomial $p(t)$. 
\end{thm}

Before we begin the proof, we collect some elementary facts regarding the Stirling numbers of the second kind. These can be found in~\cite{stanley} for example.
\begin{prop} \label{snprop}
If $n$ and $k$ are nonnegative integers, then the following hold:
\begin{eqnarray}
{n\brace k} &=& \frac{1}{k!}\sum_{j=0}^{k}(-1)^{k-j}\binom{k}{j}j^n;\label{stir1}\\
\sum_{k=0}^n{n\brace k}(x)_k &=& x^n, \textnormal{ where $(x)_k:=x(x-1)\cdots (x-k+1)$};\\
\frac{1}{k!}(e^x-1)^k &=& \sum_{m=k}^\infty {m\brace k}\frac{x^m}{m!}.
\end{eqnarray}
\end{prop}

\begin{proof}[Proof of Theorem~\ref{egf2}]
We have
\begin{eqnarray*}
(e^x+e^y-1)^\frac{1-t}{2} & = & (1+ e^x-1+e^y-1)^\frac{1-t}{2}\\
& = & \sum_{k\geq 0} {\frac{1}{2}(1-t)\choose k} (e^x-1+e^y-1)^k \\
& = & \sum_{k\geq 0} {\frac{1}{2}(1-t)\choose k}\sum_{\ell=0}^k {k\choose \ell}(e^x-1)^\ell (e^y-1)^{k-\ell} \\
& = & \sum_{k\geq 0} {\frac{1}{2}(1-t)\choose k}\sum_{\ell=0}^k k!\left(\sum_{m=\ell}^\infty{m\brace \ell}\frac{x^m}{m!}\right)\left(\sum_{n=k-\ell}^\infty{n\brace k-\ell}\frac{y^n}{n!}\right)\\
& = & \sum_{k\geq 0} \sum_{\ell=0}^k\sum_{m=\ell}^\infty\sum_{n=k-\ell}^\infty\left(\frac{1}{2}(1-t)\right)_k {m\brace \ell}{n\brace k-\ell}\frac{x^m}{m!}\frac{y^n}{n!}.
\end{eqnarray*}
Thus if we set $f(m,n)=\left[\frac{x^m}{m!}\frac{y^n}{n!}\right](e^x+e^y-1)^\frac{1-t}{2}$, then 
\begin{eqnarray*}
f(m,n)&=& \sum_{\ell=0}^m\sum_{k=\ell}^{n+\ell}\left(\frac{1}{2}(1-t)\right)_k{m\brace \ell}{n\brace k-\ell}\\
&=& \sum_{\ell=0}^m\sum_{k=0}^{n}\left(\frac{1}{2}(1-t)\right)_{\ell+k}{m\brace \ell}{n\brace k}\\
&=& \sum_{\ell=0}^m\sum_{k=0}^{n}\left(\frac{1}{2}(1-t)\right)_{\ell}\left(\frac{1}{2}(1-t)-\ell\right)_{k}{m\brace \ell}{n\brace k}\\
& = & \sum_{\ell=0}^m\left(\frac{1}{2}(1-t)\right)_{\ell}{m\brace \ell}\sum_{k=0}^n \left(\frac{1}{2}(1-t)-\ell\right)_{k}{n\brace k}\\
& = & \sum_{\ell=0}^m\left(\frac{1}{2}(1-t)\right)_{\ell}{m\brace \ell}\left(\frac{1}{2}(1-t)-\ell\right)^n.
\end{eqnarray*}
Similarly, if we set $g(m,n)=\left[\frac{x^m}{m!}\frac{y^n}{n!}\right](e^{-x}+e^{-y}-1)^\frac{-1-t}{2},$ then $$g(m,n) = \sum_{\ell=0}^m\left(\frac{-1}{2}(1+t)\right)_{\ell}{m\brace \ell}(-1)^{m+n}\left(\frac{-1}{2}(1+t)-\ell\right)^n.$$
Hence
\begin{eqnarray}
&~&
\left[\frac{x^m}{m!}\frac{y^n}{n!}\right]C(x,y,t) \nonumber \\ 
& = & \sum_{m^\prime=0}^m\sum_{n^\prime=0}^n {m\choose m^\prime}{n\choose n^\prime}\left[\sum_{\ell_1=0}^{m^\prime}\left(\frac{1}{2}(1-t)\right)_{\ell_1}{m^\prime\brace \ell_1}\left(\frac{1}{2}(1-t)-\ell_1\right)^{n^\prime}\right] \nonumber\\
&& \cdot \left[ \sum_{\ell_2=0}^{m-m^\prime}\left(\frac{-1}{2}(1+t)\right)_{\ell_2}{m-m^\prime\brace \ell_2}(-1)^{m-m^\prime+n-n^\prime} \cdot\left(\frac{-1}{2}(1+t)-\ell_2\right)^{n-n^\prime}\right] \nonumber\\
&=& \sum_{m^\prime=0}^m\sum_{\ell_1=0}^{m^\prime}\sum_{\ell_2=0}^{m-m^\prime}{m\choose m^\prime}{m^\prime\brace \ell_1}{m-m^\prime\brace \ell_2}(-1)^{m-m^\prime}\left(\frac{1}{2}(1-t)\right)_{\ell_1}\nonumber\\
&& \cdot \left(\frac{-1}{2}(1+t)\right)_{\ell_2}\left[\sum_{n^\prime=0}^n {n\choose n^\prime} \left(\frac{1}{2}(1-t)-\ell_1\right)^{n^\prime} \cdot\left(\frac{1}{2}(1+t)+\ell_2\right)^{n-n^\prime}\right]\nonumber\\
& =&\sum_{m^\prime=0}^m\sum_{\ell_1=0}^{m^\prime}\sum_{\ell_2=0}^{m-m^\prime}{m\choose m^\prime}{m^\prime\brace \ell_1}{m-m^\prime\brace \ell_2}(-1)^{m-m^\prime} \cdot \left(\frac{1}{2}(1-t)\right)_{\ell_1}\nonumber\\ 
& & \cdot \left(\frac{-1}{2}(1+t)\right)_{\ell_2}\cdot\left(\frac{1}{2}(1-t)-\ell_1+\frac{1}{2}(1+t)+\ell_2\right)^n\nonumber\\
& = &\sum_{m^\prime=0}^m\sum_{\ell_1=0}^{m^\prime}\sum_{\ell_2=0}^{m-m^\prime}{m\choose m^\prime}{m^\prime\brace \ell_1}{m-m^\prime\brace \ell_2}(-1)^{m-m^\prime}  \left(\frac{1}{2}(1-t)\right)_{\ell_1}\left(\frac{-1}{2}(1+t)\right)_{\ell_2} \nonumber\\ 
& & \cdot (1-\ell_1+\ell_2)^n.\label{eqnH}
\end{eqnarray}

Note that within the indices of summation we have the bounds $0\leq \ell_1\leq m$ and $0\leq \ell_2\leq m$ and so $1-m\leq 1-\ell_1+\ell_2\leq 1+m$. Therefore, since $h(m,n,d)=[t^d]\left[\frac{x^m}{m!}\frac{y^n}{n!}\right]C(x,y,t)$, the first conclusion in the theorem statement follows from Equation (\ref{eqnH}). 

Now notice that $1-\ell_1+\ell_2 = 1+m$ if and only if $m^\prime=0$ (and thus $\ell_1=0$) and $\ell_2=m$. Therefore 
\begin{eqnarray*}
c_{m+1}(m,d) & = & [t^d](-1)^m\left(\frac{-1}{2}(1+t)\right)_{m} \\ 
& = & (-1)^m \left(\frac{-1}{2}\right)^m [t^d](t+1)(t+3)\cdots (t+2m-1)\\
& = & 2^{-m}[t^d](t+1)(t+3)\cdots (t+2m-1).
\end{eqnarray*}

\end{proof}

The formula given in Equation (\ref{eqnH}) in the proof of Theorem~\ref{egf2} is amenable to computation in Maple. To show its utility, we give some examples. Note that $h(2,n,0)$ appears in~\cite{bln}, while $h(3,n,0)$ appears in~\cite{bll}.

\vspace{0.6cm}
\begin{tabular}{l|l}
$(m,d)$ & $h(m,n,d)$ \\
\hline 
\\
 $(2,0)$ & $\frac{3}{4}3^n-\frac{1}{2}2^n+\frac{1}{2} -\frac{1}{4}(-1)^n$\\
 \\
\hline  \\
 $(2,1) $& $3^n -\frac{1}{2}2^n$\\
 \\
\hline \\
$(2,2) $& $\frac{1}{4}3^n -\frac{1}{2}+\frac{1}{4}(-1)^n$\\ 
\\
\hline \\
$(3,0) $&${\frac {15}{8}}\,{4}^{n}-\frac{9}{4}\,{3}^{n}+{\frac {13}{8}}\,{2}^{n}-\frac{3}{4}\,
 \left( -1 \right) ^{n}+\frac{3}{8}\, \left( -2 \right) ^{n}$\\
 \\
\hline
\\ $(3,3) $&$\frac{1}{8}4^n -\frac{3}{8}2^n-\frac{1}{8}(-2)^n$\\ 
\\
\hline \\
$(4,0) $&$ {\frac {105}{16}}{5}^{n}-{\frac {45}{4}}{4}^{n}+9{3}^{n}-\frac{11}{4}{2}^{n}-\frac{5}{8}- \left( -1 \right)^{n}\left( -2 \right) ^{n}{\frac {15}{16}}\left( -3 \right) ^{n}$\\
\\
\hline \\
$(4,4) $& $\frac{1}{16}5^n-\frac{1}{4}3^n+\frac{3}{8}-\frac{1}{4}(-1)^n+\frac{1}{16}(-3)^n$\\ 
\\
\hline \\
$(5,0) $&${\frac {945}{32}}{6}^{n}-{\frac {525}{8}}{5}^{n}+{\frac {2025}{32}}\,{4}^{n}-(30){3}^{n}+{\frac {23}{16}}{2}^{n}+{\frac {225}{32}}\left( -2 \right) ^{n}-{\frac {75}{8}}\left( -3 \right) ^{n}+{\frac {105}{32}}\left( -4 \right) ^{n}$\\
\end{tabular}
\vspace{0.6cm}

\begin{thm} \label{egf2cor}
Fix a positive integer $m$ and let $a(d)=[t^d](t+1)(t+3)\cdots (t+2m-1)$. Then the proportion of $d$-dimensional $\mathcal{H}$-strata in $\Oq$ tends to $a(d)/(m!2^m)$ as $n\rightarrow\infty$.
\end{thm}

\begin{proof}
Note that for fixed $m$, it is easily seen from Equation (\ref{stir1}) that the Stirling number of the second kind satisfies ${n\brace m} \sim m^n/m!$ as $n\rightarrow \infty$. From this we deduce that for $k<m$, ${n\brace k}/{n\brace m} \rightarrow 0$ as $n\rightarrow \infty$. Now recall that, for $n\geq m$, the total number of $\mathcal{H}$-primes is the poly-Bernoulli number \cite[Corollary 1.5]{launois1} and so we deduce from \cite[Theorem 2]{ak} that the total number of  $\mathcal{H}$-primes in $m\times n$ quantum matrices is equal to  $$B_n^{(-m)} = \sum_{k=0}^m (k!)^2{n+1\brace k+1}{m+1\brace k+1}.$$ Thus as $n\rightarrow \infty$, we have
\begin{eqnarray*}
B_n^{(-m)} &\sim & (m!)^2{n+1\brace m+1}\\
& \sim & (m!)(m+1)^n.
\end{eqnarray*}
Therefore, by Theorem~\ref{egf2}, we have
\begin{eqnarray*}
\lim_{n\rightarrow \infty} \frac{\textnormal{number of $d$-dimensional $\mathcal{H}$-strata in $\Oq$}}{\textnormal{total number of $\mathcal{H}$-strata in $\Oq$}} &=& \frac{a(d)}{m!2^m}.
\end{eqnarray*}
\end{proof}

A more sophisticated asymptotic analysis would be of interest. In particular, we pose the following question.  

For a fixed positive integer $d$, does $$\lim_{n\rightarrow\infty} \frac{\textnormal{ number of $d$-dimensional $\C{H}$-strata in $\C{O}_q(M_{n,n}(\mathbb{K}))$}}{\textnormal{total number of $\C{H}$-strata in $\C{O}_q(M_{n,n}(\mathbb{K}))$}}$$ exist?  If so, what is its value?

\bibliography{casteels1}
\bibliographystyle{amsplain}

\end{document}